\documentclass[a4paper,12pt]{amsart}
\usepackage{amsmath}
 \usepackage{latexsym, amsthm, amscd, euscript, float, subfig}
\usepackage{amssymb}
\usepackage{ifthen}
\usepackage{graphicx}
\usepackage{float}
\usepackage{cite}
\usepackage{amsfonts}
\usepackage{amscd}
\usepackage{amsxtra}
\usepackage{color}

\newtheorem{theorem}{Theorem}[section]
\newtheorem{lemma}[theorem]{Lemma}

\theoremstyle{definition}

\newtheorem{example}[theorem]{Example}

\numberwithin{equation}{section}

\def\be{\begin{equation}}
\def\ee{\end{equation}}

\newcounter{alphabet}


\begin{document}
\bibliographystyle{amsplain}
\title[On the Second Hankel determinant of logarithmic coefficients]{ On the Second Hankel determinant of Logarithmic Coefficients
	for certain univalent functions}
\author[Vasudeavarao Allu]{Vasudeavarao Allu}
\address{Vasudevarao Allu, School of Basic Sciences, Indian Institute of Technology Bhubaneswar,
Bhubaneswar-752050, Odisha, India.}
\email{avrao@iitbbs.ac.in}
\author[Vibhuti Arora]{Vibhuti Arora}
\address{Vibhuti Arora, School of Basic Sciences, Indian Institute of Technology Bhubaneswar,
Bhubaneswar-752050, Odisha, India.}
\email{vibhutiarora1991@gmail.com}
\author[Amal Shaji]{Amal Shaji}
\address{Amal Shaji, School of Basic Sciences, Indian Institute of Technology Bhubaneswar,
Bhubaneswar-752050, Odisha, India.}
\email{amalmulloor@gmail.com}

\begin{abstract}
In this paper, we investigate the sharp bounds of the second Hankel determinant of Logarithmic coefficients  for the starlike and convex functions with respect to symmetric points in the open unit disk.
\end{abstract}

\maketitle

\section{Introduction}\label{Introduction}
Let $\mathcal{H}$ denote the class of analytic functions in the unit disk $\mathbb{D}:=\{z\in\mathbb{C}:\, |z|<1\}$. Then $\mathcal{H}$ is 
a locally convex topological vector space endowed with the topology of uniform convergence over compact subsets of $\mathbb{D}$. Let $\mathcal{A}$ denote the class of functions $f\in \mathcal{H}$ such that $f(0)=0$ and $f'(0)=1$. A function $f$ is said to be {\it{univalent}} in a domain $\Omega \subseteq \mathbb{C}$, if it is one-to-one in $\Omega$.  Let $\mathcal{S}$ 
denote the subclass of  $\mathcal{A}$ consisting of functions which are univalent ({\em i.e., one-to-one}) in $\mathbb{D}$. 
If $f\in\mathcal{S}$ then it has the following series representation
\begin{equation}\label{f}
	f(z)= z+\sum_{n=2}^{\infty}a_n z^n, \quad z\in \mathbb{D}.
\end{equation}
A function $f\in\mathcal{S}$ belongs to the class  $\mathcal{S}^*$, called {\em starlike function}, if $f(\mathbb{D})$ is a starlike domain with respect to
the origin. Moreover, a function $f\in\mathcal{S}$ is called {\em convex function} if $f(\mathbb{D})$ is a starlike domain with respect to each point. The class of such functions is denoted by $\mathcal{C}$. 
\\[3mm]

In \cite{sym}, Sakaguchi introduced the class of functions
that are starlike with respect to symmetric points.
A function $f\in \mathcal{A}$ is said to be {\em starlike with
respect to symmetric points} if for any $r$ close to $1$, $r<1$, and any $z_{0}$ on the circle $|z|=r$, the angular velocity of $f(z)$ about the point $f(-z_{0})$ is
positive at $z_{0}$ as $z$ traverses the circle $|z|=r$ in the positive direction, {\em i.e.},
$$
{\rm Re\,}\left( \cfrac{z_0f'(z_0)}{f(z_0)-f(-z_0)} \right) > 0,\quad  |z_0|=r.
$$
Denote by $\mathcal{S}_S^*$ the class of all functions in $\mathcal{S}$ which are starlike with respect
to symmetric points and, functions 
$f$ in the class $\mathcal{S}_S^*$ is characterized by
$$
{\rm Re\,}\left( \cfrac{zf'(z)}{f(z)-f(-z)} \right) > 0, \quad z \in \mathbb{D}.
$$
It is known that the functions in $\mathcal{S}_S^*$ are close-to-convex and hence are univalent (see \cite{sym}). Note that the class of functions starlike with respect to symmetric points obviously includes the classes of convex functions and odd starlike functions with respect to the origin. The notion of starlike functions with respect to $N$-symmetric points has been studied in \cite{sym}. In 2002, Nezhmetdinov and Ponnusamy  \cite{samy} proved that
$\mathcal{S}_S^*$ $\nsubseteq$  $\mathcal{S}^*$ and $\mathcal{S}^*$ $\nsubseteq$ $\mathcal{S}_S^*$. 
\\[3mm]

In 1977, Das and Singh \cite{das} defined the class of  convex functions with respect to symmetric points. A function $f$ $\in$ $\mathcal{A}$ is said to be {\em convex with respect to symmetric points} if, and only if, 
$$
{\rm Re\,}\left( \cfrac{(zf'(z))'}{(f(z)-f(-z))'} \right) > 0, \quad z \in \mathbb{D}.
$$\\[2mm]

The {\it Logarithmic coefficients} $\gamma_{n}$ of $f\in \mathcal{S}$ are defined by,
\begin{equation}\label{amal-1}
	F_{f}(z):= \log\frac{f(z)}{z}=2\sum\limits_{n=1}^{\infty}\gamma_{n}z^{n}, \quad z \in \mathbb{D}.
\end{equation}
The logarithmic coefficients $\gamma_{n}$ play a central role in the theory of univalent functions. A very few exact upper bounds for $\gamma_{n}$ seem to have been established. The significance of this problem in the context of Bieberbach conjecture was pointed by Milin\cite{milin} in his conjecture. Milin \cite{milin}
conjectured that for $f\in \mathcal{S}$ and $n\ge 2$, 
$$\sum\limits_{m=1}^{n}\sum\limits_{k=1}^{m}\left(k|\gamma_{k}|^{2}-\frac{1}{k}\right)\le 0,$$
which led De Branges, by proving this conjecture, to the proof of Bieberbach conjecture  \cite{De Branges-1985}. For the Koebe function $k(z)=z/(1-z)^{2}$, the logarithmic coefficients are $\gamma_{n}=1/n$. Since the Koebe function $k$ plays the role of extremal function for most of the extremal problems in the class $\mathcal{S}$, it is expected that $|\gamma_{n}|\le1/n$ holds for functions in $\mathcal{S}$. But this is not true in general, even in order of magnitude. Indeed, there exists a bounded function $f$ in the class $\mathcal{S}$ with logarithmic coefficients $\gamma_{n}\ne O(n^{-0.83})$ (see \cite[Theorem 8.4]{Duren-book-1983}). By differentiating \eqref{amal-1} and  the equating coefficients we obtain
\begin{equation}\label{gamma}
\begin{aligned}
& \gamma_{1}=\frac{1}{2}a_{2}, \\[2mm]
& \gamma_{2}=\frac{1}{2}(a_{3}-\frac{1}{2}a_{2}^{2}),\\[2mm]
& \gamma_{3}=\frac{1}{2}(a_{4}-a_{2}a_{3}+\frac{1}{3}a_{2}^{3}).
\end{aligned}
\end{equation}
If $f\in \mathcal{S}$, it is easy to see that $|\gamma_{1}|\le 1$, because $|a_2| \leq 2$. Using the Fekete-Szeg$\ddot{o}$ inequality \cite[Theorem 3.8]{Duren-book-1983} for functions in $\mathcal{S}$ in (1.4), we obtain the sharp estimate 
$$|\gamma_{2}|\le\frac{1}{2}\left(1+2e^{-2}\right)=0.635\ldots.$$
For $n\ge 3$, the problem seems much harder, and no significant bound for $|\gamma_{n}|$ when $f\in \mathcal{S}$ appear to be known. In 2017, Ali and Allu\cite{vasu2017} obtained the initial logarithmic coefficients bounds for close-to-convex functions. In 2020, Ponnusamy {\it et al.}  \cite{PSW20} computed the sharp estimates for the initial three logarithmic coefficients for a subclass of $\mathcal{S}^*$.  The problem of computing the bound of the logarithmic coefficients is also considered in \cite{cho,vasu-2018,Thomas-2016} for several subclasses of close to convex functions. In 2021, Zaprawa \cite{pawel} obtained the sharp bounds of the initial logarithmic coefficients $|\gamma_n|$ for functions in the classes $\mathcal{S}_S^*$ and 
$\mathcal{K}_S$.
\\[3mm]

For $q,n \in \mathbb{N}$, the Hankel determinant $H_{q,n}(f)$ of Taylor’s coefficients of function $f \in \mathcal{A}$ of the form \eqref{f} is defined by
$$
H_{q,n}(f) =
\begin{vmatrix}
	a_n & a_{n+1}  & \cdots & a_{n+q-1} \\ 
	a_{n+1} & a_{n+2}  & \cdots & a_{n+q} \\
	\vdots & \vdots &  \ddots &\vdots \\
	a_{n+q-1} & a_{n+q}  & \cdots & a_{n+2(q-1)}
\end{vmatrix}.
$$
The Hankel determinant for various order is also studied recently by several authors in different contexts;
for instance see \cite{Pom66,Pom67,ALT, sim21}.
One can easily observe that the Fekete-Szeg$\ddot{o}$ functional is the second Hankel determinant $H_{2,1}(f)$. Fekete-Szeg$\ddot{o}$  then
further generalized the estimate $|a_3 - \mu a_2
^2|$ with $\mu$ real for $f \in \mathcal{S}$ \cite[Theorem 3.8]{Duren-book-1983}. \\[3mm]

Identifying the widespread applications of logarithmic coefficients, recently, Kowalczyk and Lecko \cite{adam} together proposed the study of the Hankel determinant whose entries are logarithmic coefficients of $ f \in \mathcal{S}$, which is given by

$$
H_{q,n}(F_f/2) =
\begin{vmatrix}
	\gamma_n & \gamma_{n+1}  & \cdots & \gamma_{n+q-1} \\ 
	\gamma_{n+1} & \gamma_{n+2}  & \cdots & \gamma_{n+q} \\
	\vdots & \vdots &  \ddots &\vdots \\
	\gamma_{n+q-1} & \gamma_{n+q}  & \cdots & \gamma_{n+2(q-1)}
\end{vmatrix}.
$$
Kowalczyk and Lecko \cite{adam} obtained the sharp bound of second Hankel determinant of $F_f/2$, {\em i.e.,} $H_{2,1}(F_f/2)$ for starlike and convex functions. The problem of computing the sharp bounds of $H_{2,1}(F_f/2)$ has been considered in \cite{vibhuthi} for various subclasses of $\mathcal{S}$. \\[3mm]

Suppose that $ f \in \mathcal{S}$ given by \eqref{f}. Then the second Hankel determinant of $F_f/2$ by using \eqref{gamma}, is given by
\begin{equation}\label{hankel}
	H_{2,1}(F_f/2)=\gamma_1\gamma_3-\gamma_{2}^2=a_2a_4-a_3^2+\frac{1}{12}a_2^4.
\end{equation}
In this paper, we calculate the sharp bounds for $H_{2,1}(F_f/2)$ for functinos in the classes $\mathcal{S}_S^*$ and $\mathcal{K}_S$. We also provide  examples of functions to illustrate these results.

%

\section{Main Results}\label{sec3}
Let $\mathcal{B}_{0}$  denote the class of analytic functions $w : \mathbb{D} \rightarrow \mathbb{D}$ such that $w(0)=0$. Functions in $\mathcal{B}_{0}$ are known as Schwarz functions. A function $w \in \mathcal{B}_{0}$ can be written as a power series 
\begin{equation*}\label{schwarz}
	w(z)=\sum_{n=1}^\infty c_nz^n.
\end{equation*}
For two functions $f$ and $g$ that are analytic in a domain $\mathbb{D}$, we say that the function $f$ is {\em subordinate} to $g$ in
$\mathbb{D}$ and written as $f(z) \prec g(z)$ if there exists a Schwarz function $w \in \mathcal{B}_0$ such that
$$ f(z)=g(w(z)), \quad z \in \mathbb{D}.$$
In particular, if the function $g$ is univalent in $\mathbb{D}$, then $f \prec g$ if, and only if, $f(0)= g(0)$ and $f(\mathbb{D}) \subseteq g(\mathbb{D})$.
\\[3mm]

To prove our results, we need the following lemma for Schwarz functions.

\begin{lemma}\cite{schwarz}\label{lemma1}
	Let $w(z)=c_1z+c_2z^2+ \cdots$ be a Schwarz function. Then 
	$$ |c_1|\leq 1, \, |c_2| \leq 1-|c_1|^2, \, \mbox{ and } |c_3| \leq 1-|c_1|^2-\frac{|c_2|^2}{1+|c_1|}.$$
\end{lemma}

We obtain the following sharp bound for $H_{2,1}(F_f/2)$ for functions in the class $\mathcal{S}_S^*$.

\begin{theorem}\label{thm1}
	Let $f\in \mathcal{S}_S^*$. Then 
	$$
	|H_{2,1}(F_f/2)| \leq \frac{1}{4}.
	$$
	The inequality is sharp.
\end{theorem}
\begin{proof}
	Let $f\in \mathcal{S}_S^*$ be of the form \eqref{f}. Then by the definition of subordination there exists a Schwarz function $w(z)=\sum_{n=1}^\infty c_nz^n$ such that
	\begin{equation}\label{star1}
		\frac{2zf'(z)}{f(z)-f(-z)}=\frac{1+w(z)}{1-w(z)}.
	\end{equation}
	By comparing the coefficients on both sides of \eqref{star1} yields
	\begin{equation}\label{star3}
		\begin{aligned}
			& a_2=c_1, \\[2mm]
			& a_3=c_2+c_1^2, \\[2mm]
			& a_4=\frac{1}{2}\left(c_3+3c_1c_2+2c_1^3 \right).
		\end{aligned}
	\end{equation}
	By substituting the above expression for $a_2$, $a_3$, and $a_4$ in \eqref{hankel} and then further simplification gives
	\begin{equation}\label{star44}
		\begin{aligned}
			H_{2,1}(F_f/2)& =  \gamma_1\gamma_3-\gamma_{2}^2 \\ & = a_2a_4-a_3^2+\frac{1}{12}a_2^4 \\ & =\frac{1}{48}\left(c_1^4+6c_1c_3-12c_2^2-6c_1^2c_2\right).
		\end{aligned}
	\end{equation}
	From \eqref{star44} and  Lemma \ref{lemma1}, we obtain 
	\begin{equation}\label{star5}
		\begin{aligned}
			48|H_{2,1}(F_f/2)| & \leq |c_1|^4+6|c_1| \left(1 - |c_1|^2 - \cfrac{|c_2|^2}{1 + |c_1|}\right)+6|c_1|^2|c_2|+12|c_2|^2. 
		\end{aligned}
	\end{equation}
	Now  writing $x=|c_1|$ and $y=|c_2|$ in \eqref{star5}, we obtain
	\begin{equation}\label{h1}
		48|H_{2,1}(F_f/2)|  \leq F(x,y),
	\end{equation}
	where
	$$
	F(x,y)=x^4+6x \left(1 - x^2 - \cfrac{y^2}{1 + x}\right)+6x^2y+12y^2.
	$$
	In view of Lemma \ref{lemma1}, the region of variability of a pair $(x,y)$ coincides with the set
	$$ \Omega =\{ (x,y): 0\leq x \leq 1 , 0 \leq y \leq 1-x^2 \}.  $$
	Therefore, we need to find the maximum value of $F(x,y)$ over the region $\Omega$.
	The critical points of $F$ satisfies the conditions
	\begin{equation*}
	\begin{aligned}
	& \frac{\partial F}{\partial x} =  4x^3-18x^2+12xy -\cfrac{6y^2}{(1+x)^2}+6=0 \\
	& \frac{\partial F}{\partial y}=x^2+x^3+4y+2xy=0,
	\end{aligned}
	\end{equation*}
	which has no solution in the interior of $\Omega$. Hence the function $F(x,y)$ cannot have a maximum in the interior of $\Omega$. Since $F$ is continuous on a compact set $\Omega$, the maximum of $F$ attains boundary of $\Omega$. On the boundary of $\Omega$, we have 
	\begin{center}
		$F(x,0)=x^4-6x^3+6x \leq 2.4378 $ for $0\leq x \leq 1,$
	\end{center}
	\begin{center}
		$F(0,y)=12y^2 \leq 12$ for $0\leq y \leq 1,$
		
	\end{center}
	and
	\begin{center}
		$F(x,1-x^2)=x^4-12x^2+12 \leq 12$  for $0\leq x \leq 1.$
	\end{center}
	Thus combining all the above cases we obtain
	$$
	\max_{(x,y)\in \Omega} F(x,y)=12
	$$
	and hence from \eqref{h1} we have
	\begin{equation}\label{star4}
		|H_{2,1}(F_f/2)| \leq \frac{1}{4}.
	\end{equation}
	
	To prove the equality in \eqref{star4}, we consider the function
	$$
	f_1(z)=\cfrac{z}{1-z^2}=z+z^3+z^5+\cdots, \quad z \in\mathbb{D}.
	$$
	A simple computation shows that $f_1$ belongs to the class $\mathcal{S}_S^*$ and $|H_{2,1}(F_{f_1}/2)|=1/4$ and hence equality holds in \eqref{star4}. This completes the proof.
\end{proof}
Here we provide an example that associates to Theorem \ref{thm1}.
\begin{example}
	Consider the function
	\begin{equation*}
		f_2(z)= \cfrac{z}{1-z} = z+z^2+z^3+\cdots
	\end{equation*}
	It is easy to see that the function $f$ belongs to the class $\mathcal{S}_S^*$. It is easy to see that
	$$
	|H_{2,1}(F_{f_2}/2)|=\frac{1}{12} \leq \frac{1}{4}.
	$$
\end{example}

In the following result, we estimate the sharp bound for $H_{2,1}(F_f/2)$ for functions in the class $\mathcal{K}_S$.
\begin{theorem}\label{ThmKS}
	Let $f\in \mathcal{K}_S$ be of the form \eqref{f}. Then 
	$$
	|H_{2,1}(F_f/2)| \leq \frac{1}{36}.
	$$
		The inequality is sharp.
\end{theorem}
\begin{proof}
	Let  $f(z)= z+\sum_{n=2}^{\infty}a_n z^n$ be a function in $\mathcal{K}_S$, then there exists a Schwarz function $w(z)=\sum_{n=1}^\infty c_nz^n$ such that
	
	\begin{equation}\label{convex1}
		\frac{2(zf'(z))'}{(f(z)-f(-z))'}=\frac{1+w(z)}{1-w(z)}.
	\end{equation}
	First note that by equating coefficients in \eqref{convex1} we have,
	\begin{equation}\label{convex3}
		\begin{aligned}
			& a_2=\frac{1}{2}c_1, \\[2mm]
			& a_3= \frac{1}{3}\left( c_2+c_1^2 \right), \\[2mm]
			& a_4=\frac{1}{8}\left(c_3+3c_1c_2+2c_1^3 \right).
		\end{aligned}
	\end{equation}
	A simple computation using \eqref{hankel} gives,
	\begin{equation}
		H_{2,1}(F_f/2) = \frac{1}{2304}\left(11c_1^4+36c_1 \left(1 - c_1^2 - \cfrac{c_2^2}{1 + c_1}\right)+20c_1^2c_2+64c_2^2\right).
	\end{equation}
	Following the same method as used in the proof of Theorem \ref{thm1}, we obtain
	\begin{equation}\label{h2}
		\begin{aligned}
			|H_{2,1}(F_f/2)|\leq \frac{1}{2304}\left(11|c_1|^4+36|c_1| \left(1 - |c_1|^2 - \cfrac{|c_2|^2}{1 + |c_1|}\right)+20|c_1|^2|c_2|+64|c_2|^2\right),
		\end{aligned}
	\end{equation}
	where
	$$ 0 \leq |c_1| \leq 1 \text{ and }0 \leq |c_2| \leq 1-|c_1|^2 .$$
	Now by replacing $|c_1|$ by $x$ and $|c_2|$ by $y$ in \eqref{h2} gives
	\begin{equation}\label{h10}
		2304|H_{2,1}(F_f/2)|  \leq G(x,y),
	\end{equation}
	where
	\begin{equation*}
		G(x,y)= 11x^4+36x \left(1 - x^2 - \cfrac{y^2}{1 + x}\right)+20x^2y+64y^2.
	\end{equation*}
	In view of Lemma \ref{lemma1}, the region of variability of a pair $(x,y)$ coincides with the set
	$$ \Omega =\{ (x,y): 0\leq x \leq 1 , 0 \leq y \leq 1-x^2 \}. $$
	Thus we need to find the maximum value of $G(x,y)$ over the region $\Omega$.
	The critical points of $G$ satisfies the conditions
	$$
	\frac{\partial G}{\partial x}=44x^3-108x^2+40xy -\cfrac{36y^2}{(1+x)^2}+36=0,
	$$
	and
	$$
	\frac{\partial G}{\partial y}=5x^2+5x^3+32y+14xy=0,
	$$
	which has no solution in the interior of $\Omega$. By using the elementary calculus, we can show that the maximum of $G(x,y)$ should exists on the boundary of $\Omega$.
	It is easy to see that on the boundary line $x=0$ and $ 0 \leq y \leq 1$, we have $G(0,y)=64y^2$ and its maximum on this line is equal to $64$.
	Similarly, on the boundary line $y=0$ and  $0 \leq x \leq 1,$ we have $G(x,0)=11x^4-36x^3+36x$ and its maximum on this line is $15.512$.
	Finally, on the boundary curve $y=1-x^2$ and $0 \leq x \leq 1$, we have $G(x,1-x^2)=19x^4-72x^2+64$ and its maximum on this curve is $64$.
	Thus, combining all the above cases yields
	$$
	\max_{(x,y)\in \Omega} G(x,y)=64
	$$
	and hence from \eqref{h10} we obtain
	\begin{equation}\label{convex4}
		|H_{2,1}(F_f/2)| \leq \frac{1}{36}.
	\end{equation}
	
	For the sharpness of the inequality \eqref{convex4} we consider the function
	$$
	f_3(z)=\cfrac{1}{2}\log\cfrac{1+z}{1-z}=z+\cfrac{z^3}{3}+\cfrac{z^5}{5}+\cdots$$
	which belongs to the class $\mathcal{K}_S$. A simple computation shows that $|H_{2,1}(F_{f_3}/2)|=1/36$ and hence the inequality in \eqref{convex4} is sharp. This completes the proof.
\end{proof}

In the following example we construct a function that agree with Theorem \ref{ThmKS}.

\begin{example}
	Consider the function
	\begin{equation*}
		f_4(z)=- \log (1-z)= z+\cfrac{z^2}{2}+\cfrac{z^3}{3}+\cdots .
	\end{equation*}
A simple compuation shows that
$$
{\rm Re\,}\left( \cfrac{(zf_4'(z))'}{(f_4(z)-f_4(-z))'} \right) = \cfrac{1}{2} \, {\rm Re\,}\left( \cfrac{1+z}{1-z} \right) > 0.
$$	
and hence the function $f_4\in \mathcal{K}_S$. It is easy to see that
	$$
	|H_{2,1}(F_{f}/2)|= \frac{11}{576} \leq \frac{1}{36}.
	$$
\end{example}

\bigskip
\noindent
{\bf Acknowledgment.}
The first author thanks SERB-CRG, the second author thanks IIT Bhubaneswar for providing Institute Post Doctoral Fellowship, and the third author's research work is supported by CSIR-UGC.

\end{document}